\documentclass[sn-mathphys]{sn-jnl}

\jyear{2021}%
\usepackage{enumerate}
\usepackage[T1]{fontenc}
\usepackage{amsthm}
\usepackage{array}

\theoremstyle{thmstyleone}%
\newtheorem{theorem}{Theorem}

\newtheorem{cor}[theorem]{Corollary}
\newtheorem{prop}[theorem]{Proposition}

\newtheorem{proposition}[theorem]{Proposition}%

\theoremstyle{thmstyletwo}%
\newtheorem{remark}{Remark}%
\theoremstyle{thmstylethree}%
\newtheorem{definition}{Definition}%

\begin{document}

\title[Global and nonglobal solutions for an abstract semilinear heat equation]{Global and nonglobal solutions for an abstract semilinear heat equation}

\author*[1]{\fnm{Ricardo} \sur{Castillo}}\email{rcastillo@ubiobio.cl}

\author[2]{\fnm{Ricardo} \sur{Freire}}\email{ricardo.silva@uesb.edu.br}
\equalcont{This manuscript was submitted for publication on November 11, 2024.}

\author[3]{\fnm{Miguel} \sur{Loayza}}\email{miguel.loayza@ufpe.br}
\equalcont{This manuscript was submitted for publication on November 11, 2024}

\affil*[1]{\orgdiv{Departamento de Matem\'atica, Facultad de Ciencias}, \orgname{Universidad del B\'io-B\'io}, \orgaddress{\street{Avenida Collao 1202, Casilla 5-C}, \city{Concepción}, \state{B\'{\i}o-B\'{\i}o}, \country{Chile}}}

\affil[2]{\orgdiv{Departamento de Matem\'atica}, \orgname{Universidade Estadual do Sudoeste da Bahia - UESB}, \orgaddress{ \city{Vitória da Conquista}, \state{Bahia}, \country{Brasil}}}


\affil[3]{\orgdiv{Departamento de Matem\'atica}, \orgname{Universidade Federal de Pernambuco}, \orgaddress{ \city{Recife}, \state{Pernambuco}, \country{Brasil}}}


\abstract{In this paper, we focus on the existence of both global and non-global nonnegative solutions for the abstract semilinear heat equation $u_t - Lu = \sum_{i=1}^k \ell_i(t) f_i(u)$ in $L^p(X, \mu)$. Here $(X,\mu)$ represents a $\sigma$-finite measure space, $L$ is the infinitesimal generator of a sub-Markovian strongly continuous semigroup of bounded linear operators in $L^p(X,\mu)$, and for each \(i = 1, \ldots, n\), the functions \(\ell_i: [0, \infty) \rightarrow [0, \infty)\) and \(f_i: [0, \infty) \rightarrow [0, \infty)\) are continuous. We utilize our findings to apply to the Laplacian operator on sub-Riemannian manifolds and locally finite graphs obtaining Fujita-type results. Our results extended the recent work of Lenz, Schmidt, and Zimmermann referenced in \cite{ZIMMERMANN}.}


\maketitle

\section{Introduction} 


Let $(X,\mu)$ a $\sigma$-finite measure space, $p \in [1,\infty)$ and $L: D(L) \rightarrow L^p(X, \mu)$ be a closed linear operator in $L^p(X, \mu)$ generating a strongly continuous sub-Markovian semigroup $\{S_L(t) \}_{t \geq 0}$ of bounded linear operators in $L^p(X, \mu)$.  For further information regarding semigroup theory see the books \cite{Engel} and \cite{PAZY}. The space
$$ L^p(X,\mu) :=  \left\{ \varphi: X \rightarrow \mathbb{R} ;  \int_{X} \vert \varphi \vert^{p} (s) d\mu(s)  < \infty  \right\} $$
is equipped with the norm $ \| \varphi \|_p :=  \left( \int_{X} \vert \varphi \vert^{p} (s) d\mu(s) \right)^{1/p}$ for $1\leq p<\infty$. When $p=\infty$, we define
$$ L^\infty(X,\mu) :=  \left\{ \varphi: X \rightarrow \mathbb{R};  \varphi \mbox{ is measurable and $\mu$-essentially bounded }  \right \}, $$
with its usual norm 
$ \| \varphi \|_\infty $. The set $L^p_{+}(X,\mu)$ denotes the subset of $L^p(X,\mu)$ consisting of functions that are nonnegative $\mu$-almost everywhere (a.e.).

A significant theory about these semigroups, mentioned at the outset, is well-documented (e.g., \cite{VAROPOULOS1}, \cite{Rocner}, \cite{JACOB}, \cite{LISKEVICH}, \cite{LISKEVICH1}, \cite{FARKAS} and the references therein). Specifically, it is widely known that when $L$ is a positive self-adjoint operator on $L^2(X, \mu)$ or a Dirichlet operator on $L^2(X, \mu)$, the corresponding semigroup $\{S_L(t)\}_{t\geq 0}$ is sub-Markovian.  The extendibility of such semigroups to $L^\infty(X,\mu)$ have been examined in \cite{LISKEVICH} and \cite{JACOB}, while the existence of the kernel operator associated with these semigroups were established in \cite{JACOB}.

We consider the following abstract semilinear parabolic problem
\begin{equation}\label{eq2.1}
 \left\{\begin{array}{rll}
 u_{t} - L u   = & \sum_{i=1}^k \ell_i(t) f_i(u) & (x,t) \in  X \times (0,T),\\
 u(0)= & u_{0}\geq 0 & x \in X ,\\
 u & \geq 0,
      \end{array}\right.
\end{equation}
where $k\in \mathbb{N}$, $u_{0}\in L^p_+(X, \mu)$,  $ \ell_i: [0, \infty) \rightarrow [0,\infty)$ is continuous, and $ f_i: [0, \infty) \rightarrow [0,\infty)$ is locally Lipschitz continuous verifying $f_i(0)=0$ for $i=1,\cdots, k$. Additionally, to obtain global and non-global existence results we assume the following conditions:
\begin{itemize}
\item[($H_1$)] There exists $\lambda>0$ such that the function $s\in (0,\lambda] \mapsto f_i(s)$ is nondecreasing on $(0, \lambda]$ and $\limsup_{s\to 0^+} f_i(s)/s <\infty$ for every $i=1,\cdots, k$.\\

\item[($H_2$)] There exists $j \in \{1, \cdots, k\}$ such that $f_j$ is convex, strictly positive on $(0,\infty)$ and satisfies $ \displaystyle \int_{\tau}^{\infty} \frac{d \sigma}{f_j(\sigma)}  < \infty $ for all $\tau > 0$.
\end{itemize}
Condition $(H_1)$ is verified for $f(t) = t^{p}$, where $p > 1$ and $t \geq 0$. More generally, this is true when $f: [0, \infty) \rightarrow [0,\infty)$ is a non-decreasing convex function with $f(0) = 0$. Another example is the concave function $f(t)$, defined by $f(t)=\sin(t)$ for $t\in [0,\pi/2]$ and $f(t)=1$ for $t>\pi/2$. On the other hand, the integral condition in $(H_2)$ is related to the Osgood-type condition (see, e.g. \cite{Laister1}) which assures that the solution of the ordinary differential equation $u' =f(u)$ is defined on a finite interval. The sum in (\ref{eq2.1}) allows us to combine different kinds of nonlinearities as treated in \cite{HUI}.

Recently, Lenz, Schmidt, and Zimmermann in \cite{ZIMMERMANN} obtained a sufficient condition for the non-existence of global nonnegative mild solutions to the abstract problem
\begin{equation}\label{SingleAbstract}
\left \{ \begin{array}{rll}
 u_t-  L u & =  f(u)& \mbox{ for } t \in (0,T),\\
u(0)& = u_0,  & \\
u &\geq 0,   &
\end{array}\right.
\end{equation}
where $u_0 \in L^p_{+}(X,\mu),$ $f: [0,\infty) \rightarrow [0,\infty)$ is a continuous and convex function (thus locally Lipschitz continuous), and the operator $L$ satisfies the same conditions as stated in the above lines. They showed the following theorem:
\begin{theorem}[\cite{ZIMMERMANN}]\label{AbstractTHEO} Let $u_0 \in L^p_{+}(X,\mu).$ Assume that $f$ is a nonnegative continuous convex function such that $f(0)=0$, $f(\sigma)>0 $ for all $\sigma \in (0, \infty)$ and  $ \displaystyle \int_{z}^{\infty} \frac{d \sigma}{f(\sigma)} < \infty $ for all $z > 0.$ If there exists a time $T > 0$ and a
measurable subset $K$ of $X$ with $\mu(K) \in (0,\infty)$ such that
$$ \displaystyle \int_{\frac{1}{\mu(K)} \int_{K} S_{L}(T) u_0 d\mu}^{\infty} \frac{d \sigma}{f(\sigma)} \leq T.  $$
Then there is no global mild solution of problem \eqref{SingleAbstract}.
\end{theorem}
In the proof, the authors used a mild solution formulation to deduce a suitable ordinary differential inequality. It is important to note that this technique was introduced by Weissler in \cite{Weissler1,Weissler2} and by Fujita in \cite{Fujita1} (Fujishima and Ioku recently applied this method successfully in \cite[Lemma 4.1]{Fujishima1}). Moreover, they have utilized Theorem \ref{AbstractTHEO} in various interesting scenarios, such as the Laplacian on Riemannian manifolds (see \cite[Theorem 5.3]{ZIMMERMANN}), graphs (see \cite[Theorem 5.7]{ZIMMERMANN}), and metric measure spaces with a heat kernel (see \cite[Theorem 5.9]{ZIMMERMANN}).  

When $L=\Delta$ is the Laplacian operator,  Fujita (see \cite{FU}) studied the heat equation
\begin{equation}\label{eq2.1f}
\left\{ \begin{array}{rll}
 u_t - \Delta u & =  u^q& \mbox{ in }  \mathbb{R}^N \times (0,T),\\
u(0) &= u_0  & \mbox{ in }  \mathbb{R}^N,
\end{array}\right. 
\end{equation}
and discovered the critical exponent $q^\star= 1 + 2/N$, which determines the blow-up and global existence of the solutions, that is,
\begin{itemize}
  \item If $q>q^\star$, then problem (\ref{eq2.1f}) possesses global nonnegative nontrivial solutions.
  \item If $1 < q < q^\star$, then all nontrivial nonnegative solutions to problem (\ref{eq2.1f}) blow up in finite time.
\end{itemize}


Since this important result was obtained numerous researchers have worked on obtaining Fujita critical exponents considering various settings. For instance, see \cite{QS} for results in the Euclidean scenario. For results involving Riemannian manifolds and locally finite graphs, refer to \cite{BANDLE}, \cite{Zhang1}, \cite{Zhang2}, as well as \cite{HUI}, \cite{LIU}, \cite{GEORGIEV}, \cite{ROSSI}, \cite{Wu1}, \cite{Monticelli1}, and \cite{Lin1}.
Motivated by these results our main objective is to provide a unified approach to obtain global and non-global solutions considering various geometries. In particular, we recover and derive new Fujita-type results considering sub-Riemannian manifolds and locally finite graphs.

Our main result is the following.
\begin{theorem}\label{thm1} Suppose that  $ \ell_i: [0, \infty) \rightarrow [0,\infty)$ is continuous and $f_i : [0,\infty) \rightarrow [0,\infty)$ is locally Lipschitz continuous such that $f_i(0)=0 $ for every $i = 1,\cdots,k.$
\begin{itemize}
 \item[(i)] Assume that the function \( f_i \) satisfies condition \( (H_1) \) for each \( i = 1, \ldots, k \), and there exists a nontrivial $w_0 \in L^p_{+}(X, \mu) \cap L^{\infty}(X, \mu) $ satisfying $ \|w_0\|_\infty \leq \lambda $ and 
 \begin{equation}\label{GLOB1}
 \displaystyle \sum_{i=1}^k\int_{0}^{\infty} \ell_i(\sigma) F_i \left(\|S_L(\sigma)w_0\|_{\infty} \right) d\sigma   < \frac{\beta}{\beta+1}
 \end{equation}
for some $\beta> 0$, where $F_i(\sigma)=\sup_{0<t\leq \sigma} f_i(t)/t$ when $t>0$ and $F_i(0)= \limsup_{s\to 0^+} f_i(s)/s$. Then there exists $ \delta > 0 $ depending on $\beta$ such that for $ u_0 =\delta w_0 $, the solution of problem \eqref{eq2.1} is a global solution.

 \item[(ii)] Let $u_0 \in L^p_{+}(X, \mu)$, $u_0 \neq 0$. Assume that condition $(H_2)$ is verified for some $j_0 \in \{1,\cdots k\}$, $\ell_{j_0}(t)>0$ for $t>0$, and there exists a measurable subset $K$ of $X$ with $\mu(K) \in (0,\infty)$ such that
  \begin{equation} \label{TTa}
  \int_{ \frac{1}{\mu(K)}\int_K S_L ( T ) u _0 d \mu}^{\infty} \frac{d \sigma}{f_{j_0}(\sigma)} \leq \int_{0}^{T} \ell_{j_0}(\sigma) d \sigma,
  \end{equation}
for some $T>0$. Then the solution of problem \eqref{eq2.1} with initial condition $u_0$ is non-global.
\end{itemize}
\end{theorem}
\begin{remark} ~~
\begin{itemize} 
\item[(i)] The locally Lipschitz continuity condition is used to guarantee the local existence of solutions, as seen in Proposition \ref{LOCALEXIST}. 

\item[(ii)] Condition (\ref{GLOB1}) extends the one found in \cite[Theorem 3 (a)]{Weissler} where it was treated the case $f_i(t)=t^{q}, q>1, t\geq 0, \ell_i=1$ and $u_0\in L^p_{+}(\mathbb{R}^N)$. 

\item[(iii)] We need the additional assumption $u_0\in L^\infty(X, \mu)$ due to the absence of smoothing effect of the sub-Markovian semigroup $\{S_L(t)\}_{t\geq 0}$ which is verified when $L=\Delta$ (see \cite[Proposition 1]{Weissler}). Moreover, the assumption $u_0\in L^p_+(X,\mu)\cap L^\infty(X,\mu)$ was used to establish the existence of a locally essentially bounded mild solution to problem \eqref{SingleAbstract} (as detailed in Section 3 of \cite{ZIMMERMANN}). Therefore, it makes sense  consider similar constraints on the initial data for the problem \eqref{eq2.1}.

\item[(iv)] Theorem \ref{thm1} (ii) broadens the scope of Theorem \ref{AbstractTHEO} by incorporating general time-weighted sources $\ell_i \, (i=1,\cdots,k)$.
\end{itemize}
\end{remark}


We use our main result to analyze the problem (\ref{eq2.1}) considering a sub-Riemannian manifold (Homogeneous Carnot groups), a Riemannian manifold (The hyperbolic space) and locally finite graphs. In all the cases we consider polynomial, logarithmic, and exponential nonlinearities.  Thus, we will be able to identify known Fujita exponents, discover new ones, and reveal new global and non-global results, see Section \ref{Section3}. In summary, we will demonstrate the following:
\begin{itemize}
    \item[(a)] Let $X=\mathbb{G}$ be a Carnot group with  homogeneous dimension $Q$. If either $f_i (u) = u^{q_i}$ or $f_i (u) = (1+u)[\ln(1+u)]^{q_i}$, $q_i> 1$ and $\ell_i(t) = t^{\alpha_i} \, (\alpha_i>-1)$, for $i = 1, \cdots, k$. Then, the value $$q^* = \max_{1 \leq i \leq k} \left\{ 1 + \frac{2+2 \alpha_i}{Q} \right\},$$is the critical exponent in the Fujita sense.  (see, Corollary \ref{coEXTRA}). We also analyze the case  $f_i (u) = e^{ \beta_i u} -1$, $\beta_i>0$ for $i \in \{1, \cdots, k\}$ (see Corollary \ref{coEXP}). It is worth mentioning that if $\mathbb{G}= \mathbb{R}^N$ is the $N$-dimensional Euclidean space then $Q=N $; and if $\mathbb{G} = \mathbb{H}^N$ is the Heisenberg group then $Q=2N+2$. 
    
    \item[(b)]  If $X = \bold{H}^N$ is the $N$-dimensional hyperbolic space, $f_i (u) = u^{q_i}$ or $f_i (u) = (1+u)[\ln(1+u)]^{q_i}$, $q_i>1$,  for every $i=1, \cdots,k$, and $\ell_i(t ) = e^{\alpha_i t} \, (\alpha_i>0)$, the Fujita critical exponent is given by 
    $$q^* = \max_{1 \leq i \leq k} \left\{ 1 + \frac{4 \alpha_i}{(N-1)^2} \right\}, $$
    see Corollary \ref{co2}.
\item[(c)]   Similar results to those presented in items (a) and (b) hold if \(X = G = (V,E)\) is a simple, locally finite, connected graph satisfying the \(CDE'(n,0)\) condition along with a given volume growth condition (see Corollary \ref{co-graf1} for more details).

\end{itemize}


The remainder of the paper is structured as follows. Section \ref{prel} provides the necessary preliminaries. The proof of Theorem \ref{thm1} is presented in Section \ref{glob-non-glob}. Lastly, Section \ref{Section3} discusses some applications.

\section{Preliminaries}\label{prel}

Note that for every $i=1\cdots k$, $f_i: [0,\infty) \to [0,\infty)$ is given. For simplicity, we will also use $f_i$ to denote the operator of composition with $f_i$, defined by
\begin{eqnarray*}
D(f_i) & := & \left\{ \varphi \in L^p_{+}(X, \mu);  f_i \circ \varphi \in  L^p(X, \mu) \right\} \\
f_i(\varphi) &:= & f_i \circ \varphi .
\end{eqnarray*}

The concept of solution of problem (\ref{eq2.1}) is the following.
\begin{definition} \label{MILD1}  Let $u_0\in L^{p}_{+}(X,\mu)$. A function $u:[0,T) \rightarrow L^{p}_{+}(X,\mu)$ is a mild solution (or simply a solution) of problem (\ref{eq2.1}) in $[0,T) $, if for $i=1,...,k$,  $u(t)\in D(f_i)$ for all $t \in (0,T)$, $f_i \circ u \in L^1((0,t), L^p(X,\mu)\ )$ for every $t \in (0,T)$, and $u$ satisfies the variation of constants formula
\begin{equation*}
\begin{array}{rll}
u(t) &=& S_{L}(t)u_0 + \displaystyle \sum_{i=1}^{k} \displaystyle  \int_{0}^{t} S_{L}(t - \sigma)\ell_i(\sigma)f_i(u(\sigma))d\sigma, 
\end{array}
\end{equation*}
for all $t \in [0,T)$. { This solution is said to be maximal if it cannot be extended to an interval that is strictly larger than $[0,T)$. In this case, we write $T=T_{\max}$}.  Also, if $u$ is a mild solution of (\ref{eq2.1})  in $[0,\infty)$, then we call $u$ a global solution of (\ref{eq2.1}).
\end{definition}

\begin{definition}[Locally essentially bounded solution] A solution $u$ of problem (\ref{eq2.1}) is said to be locally essentially bounded if $u(t) \in L^{\infty}(X,\mu)$ for every $t \in [0,T_{\max})$ and $\sup_{t\in [0,T]} \{\|u(t)\|_{\infty}\}$ is finite for every $T \in [0,T_{\max})$.
\end{definition}

The following local existence result can be obtained following similar ideas of \cite[Theorem 3.1]{ZIMMERMANN}.
\begin{prop}{(Local existence)} \label{LOCALEXIST}  Let $\ell_i \in C[0,\infty)$ satisfying $\displaystyle\sum_{i=1}^k \lim_{t \rightarrow 0^+} \displaystyle \int_0^t \ell_i(\sigma) d\sigma = 0$ and $f_i \in C[0,\infty)$ locally Lipschitz continuous with $f_i(0)=0$ for $i=1, \cdots, k$. Then for every $u_0 \in L^p_+(X, \mu) \cap L^{\infty}(X, \mu) $, $u_0 \neq 0$, there exists  a unique maximal locally essentially bounded mild solution $u$ of problem \eqref{eq2.1} defined on $[0,T_{\max})$, for some $0<T_{\max} \leq \infty$. Moreover, if $T_{max} < \infty $, then we have
$$ \lim_{t \rightarrow T_{\max}} \|u(t)\|_{\infty} = \infty.$$
\end{prop}

\section{Global and non-global existence}\label{glob-non-glob}

In this section, we prove Theorem \ref{thm1}.

\begin{proof}[Proof of Theorem \ref{thm1} (i).]  Assume that condition (\ref{GLOB1}) is satisfied.

Take $\delta=1/(1+\beta)$ and  $u_0 = \delta w_0 \in   L^p_{+}(X, \mu) \cap L^{\infty}(X, \mu) $. Now we define the sequence $\{ u^n\}_{n \in \mathbb{N}}$ by $u^0(t) = S_L(t)u_0 $ and
\begin{equation} \label{GLOB2}
u^n(t) = S_L(t)u_0 + \displaystyle\sum_{i=1}^k\displaystyle \int_0^t S_L(t - \sigma) \ell_i(\sigma)f_i(u^{n-1}(\sigma)) d \sigma
\end{equation}
for $t\geq 0.$

We claim that 
\begin{equation} \label{Claim11}
u^{n}(t) \leq (1+\beta) S_L(t)u_0
\end{equation}
for $n \geq0 $ and $t\geq 0.$ Indeed,  by induction on $n$, we see \eqref{Claim11} is clear for $n=0 $. Consider now that \eqref{Claim11} holds for some $n \in \mathbb{N}$. Since 
$$  (1+ \beta)S_L(t)u_0  = S_L(t)w_0 \leq \lambda S_L(t)  \dfrac{w_0}{ \| w_0 \|_{\infty}}   \leq \lambda$$ 
(by sub-Markovian property), we have
$ \| (1+ \beta)S_L(t)u_0 \|_{\infty} = \| S_L(t)w_0 \|_{\infty} \leq  \lambda$ 
for $t>0$. Due to the property \( (H_1) \) and since \( f_i(\sigma) \leq F_i(\sigma) \sigma \) for \( \sigma \geq 0 \), it follows that:
\begin{eqnarray*}
u^{n+1}(t) & \leq & S_L(t)u_0 + \displaystyle\sum_{i=1}^k\displaystyle \int_0^t S_L(t - \sigma)  \ell_i(\sigma)f_i(u^{n}(\sigma) )  d \sigma\\
 & \leq &
 S_L(t)u_0  + \displaystyle\sum_{i=1}^k \int_0^t \ell_i(\sigma) S_L(t - \sigma)  f_i((1+\beta) S_L(\sigma)u_0) d \sigma\\
& \leq &  S_L(t)u_0 +  \displaystyle\sum_{i=1}^k \int_0^t  \ell_i(\sigma) S_L(t - \sigma)  F_i((1+\beta) S_L(\sigma)u_0)   \left[(1+\beta) S_L(\sigma)u_0 \right]  d \sigma
 \\
& \leq  &  S_L(t)u_0 + (1+\beta) S_L(t)u_0  \displaystyle\sum_{i=1}^k \int_{0}^{t} \ell_i(\sigma) F_i \left(\| S_L(\sigma)w_0 \|_{\infty} \right )d\sigma \\
& \leq &  S_L(t)u_0 + (1+\beta) S_L(t)u_0 ~\frac{\beta}{\beta + 1}  \leq (1 + \beta)  S_L(t)u_0.
\end{eqnarray*}
Thus, (\ref{Claim11}) holds for all $n\geq 0$.

Note that estimate (\ref{Claim11}) and the fact that the functions $f_i \,\,  (i=1,...,k)$ are nonnegative and nondecreasing imply that the sequence $\{ u_n(\sigma)\}_{n \in \mathbb{N}}$ is nondecreasing and dominated in $L^p(X, \mu)$ by $(1+\beta) S_L(t)u_0$. Therefore, according to the dominated convergence theorem, there exists
$$ U(\sigma) = \lim_{n \to \infty} u^n(\sigma)  \, \, \mbox{ in } L^p(X, \mu), $$
for $\sigma \in [0, \infty).$ 

We will now show that $U: [0, \infty) \rightarrow L^p(X, \mu)$ is the desired global solution.  Note that, $U(\sigma) \in L^p(X, \mu) \cap L^{\infty}(X, \mu) $ for every $\sigma>0$. In fact, from (\ref{Claim11}) we have
$$ \| u^n(t)\|_{\infty} \leq \| S_L(t)w_0 \|_{\infty} \leq \| w_0 \|_{\infty}< \lambda, \mbox{ for } t > 0, $$
(so $\sup_{t>0} \|U(t) \|_{\infty}\leq  \lambda$). 
From this, and since $f_i$ is locally Lipschitz (hence Lipschitz on the interval $(0,\lambda]$ with Lipschitz constant $M_i>0$), we have
$$
    \begin{array}{ll}
      & \left \|  S_L(t - \sigma) \, \ell_i (\sigma)[f_i(u^{n}(\sigma)) -f_i(U(\sigma)) ]  \right\|_p \\ 
      &  \leq M_i \, \ell_i(\sigma) \left( \displaystyle \sup_{0<\sigma<t}\|S_L(\sigma)\| \right)  \|u^n(\sigma) - U(\sigma)\|_p \to 0
    \end{array}
$$
when $n \rightarrow \infty $, for $\sigma \in (0,t)$ and $i=1,..,k$. Hence 
$$\lim_{n \rightarrow \infty} S_L(t - \sigma) \, \ell_i(\sigma) \, f_i(u^n(\sigma))  =  S_L(t - \sigma) \, \ell_i(\sigma) \, f_i(U(\sigma))  \mbox{ in } L^p(X, \mu),$$
for $\sigma \in (0,t)$ and $i=1,...,k$. Also, from (\ref{Claim11}) and condition (\ref{GLOB1}), the functions $$ \sigma \mapsto S_L(t - \sigma) \, \ell_i(\sigma) \, f_i(u^n(\sigma)) \in L^{1}((0,t),L^p(X, \mu)),$$ for $i=1,...,k$, and are dominated by the functions $$ \sigma \mapsto (1+\beta) \, \ell_i(\sigma) \, S_L(t)u_0 \, F_i(\| S_L(\sigma)w_0 \|_\infty ) \in L^{1}((0,t),L^p(X, \mu)  ),$$ 
for $i=1,...,k$. Therefore, by the dominated convergence theorem for $L^p(X, \mu)$-valued functions, we obtain
$$ \lim_{n \rightarrow \infty} \sum_{i=1}^{k}\int_{0}^{t} S_L(t - \sigma) \, \ell_i(\sigma) \, f_i(u^n(\sigma)) d\sigma = \sum_{i=1}^{k}\int_{0}^{t} S_L(t - \sigma) \, \ell_i(\sigma) \, f_i(U(\sigma)) d\sigma,$$
and the result follows letting $n \to \infty$ in the equation (\ref{GLOB2}).
\end{proof}

\begin{proof}[Proof of Theorem \ref{thm1} (ii).]

Assume that condition (\ref{TTa}) holds. We argue by contradiction. Suppose that there exists a nontrivial mild solution $u$ defined in $[0, \infty)$ with initial condition $u(0) \neq 0$, $u(0) \in L^p_{+}(X, \mu)$. Let $\nu = [1/\mu(K)] \, \mu $ with $0< \mu(K) < \infty$. For any $\tau>0$, define the function $\Theta:[0,\tau] \rightarrow \mathbb{R}$ by
$$\Theta( t ) = \int_K S_L ( \tau - t ) u ( t ) d \nu,$$
for $0\leq t\leq \tau$. From the definition of a mild solution \ref{MILD1}, we derive the following equality based on the fact that the bounded linear operator $S_L(t)$ commutes with Bochner integrals \cite[Theorem 3.1.7]{HSING} and the Fubini theorem \cite[Theorem 8.8]{RUDIN}.
\begin{equation}\label{eq-blo1}
    \Theta(t+h) - \Theta(t) 
 = \int_t^{t+h}\sum_{i=1}^k  \ell_i(\sigma)  \left[ \int_K  S_L(\tau  -\sigma ) f_i(u(\sigma)) d \nu \right] d\sigma.
\end{equation}

Due to \cite[Lemma 2.2]{ZIMMERMANN}, the function $\psi(\tau):= \int_{\tau}^{\infty} 1/f_{j_0}(s) ds$ is bijective and continuously differentiable with $\psi'(\tau)<0$ (and so $[\psi^{-1}]'(\tau)<0$), therefore by condition (\ref{TTa}),
$$\Theta(0) = \int_K S_L ( \tau ) u_0 d \nu >  \psi^{-1} \left( \int_{0}^{\tau} \ell_{j_0}(\sigma) ~ d\sigma \right)>0.$$
Thus, by the equation \eqref{eq-blo1}, $ \Theta(t) > 0$ for $t\in [0,\tau]$.

Additionally, it should be noted that due to equality \eqref{eq-blo1}, $\Theta$ is absolutely continuous and thus differentiable almost everywhere in $[0,\tau]$ with
$$ \Theta'(t) = \sum_{i=1}^k \ell_i(t) \int_K S_L(\tau - t) f_i(u(t)) \, d\nu, $$
(see e.g. Theorem 7.18 in \cite{RUDIN}). 
Now, using Jensen inequality (see \cite[Lemma 4.4]{ZIMMERMANN}) we obtain
\begin{multline*}
   \Theta'(t) \geq \ell_{j_0}(t)  \int_K  S_L(\tau  - t ) f_{j_0}(u(t)) d \nu \geq \ell_{j_0}(t)  \int_K  f_{j_0}\left( S_L(\tau  - t ) u(t) \right) d \nu \\
   \geq \ell_{j_0}(t) \, f_{j_0}\left( \int_K   S_L(\tau  - t ) u(t)  d \nu \right) = \ell_{j_0}(t) f_{j_0}(\Theta(t)), 
\end{multline*}
for almost all $t \in [0, \tau]$. In particular, it follows that $\Theta$ is monotonically increasing function on $[0, \tau]$, since $\ell_{j_0}(t), f_{j_0}(t) > 0$ for $t>0$. Hence, 
\begin{eqnarray*}
\int_{ \int_K S_L ( \tau ) u _0 d \nu}^{\infty} \frac{d \sigma}{f_{j_0}(\sigma)} &\geq &   \int_{ \Theta(0) }^{\Theta(\tau)} \frac{d \sigma}{f_{j_0}(\sigma)} =  \int_{ \Theta(0) }^{\infty} \frac{d \sigma} {f_{j_0}(\sigma)} -  \int_{ \Theta(\tau) }^{\infty} \frac{d \sigma}{f_{j_0}(\sigma)}  \\
&=& - \int_{0}^{\tau}  \left(  \int_{ \Theta(t) }^{\infty} \frac{d \sigma}{f_{j_0}(\sigma)} \right)'dt  \\\
&=& \int_{0}^{\tau} \frac{\Theta'(t)}{f_{j_0} \left( \Theta(t) \right)}~ dt ~ \geq ~ \int_{0}^{\tau} \ell_{j_0}(t) ~ dt.
\end{eqnarray*}
This contradicts condition \eqref{TTa}.

\end{proof}

\section{Applications and Fujita critical exponents }\label{Section3}

\subsection{Applications on sub-Riemannian manifolds}

We begin by reviewing fundamental concepts of sub-Riemannian geometry. For more details,  see \cite{BONFIGLIOLI}, \cite{Ovidiu}, and \cite{Davide}.

A sub-Riemannian manifold is a real manifold $(M,\mathcal{H},g)$ of dimension $N$ together with a nonintegrable distribution $\mathcal{H}$ of rank $k$ (where $k \leq N$ and $\mathcal{H}$ is a vector subbundle of $TM$). It is equipped with a sub-Riemannian metric $g$, which is an application $g_p : \mathcal{H}_p \times \mathcal{H}_p \rightarrow [0 , \infty)$ for all $p \in M$ (where $\mathcal{H}_p$ is a vector subspace of $T_pM$). This application is a positive definite, nondegenerate, inner product.

A curve $\gamma: [0, T] \rightarrow M$ is considered a horizontal curve if $\gamma'(t) \in \mathcal{H}_{\gamma(t)}$ for all $t \in [0, T]$. The metric $g$ can also be used to define the length of horizontal curves $\gamma: [0, T] \rightarrow M$ using the formula
$$L(\gamma) = \displaystyle \int_0^T [g(\gamma'(t), \gamma'(t) )]^{1/2} dt. $$

The sub-Riemannian Carnot-Caratheodory distance between two points $P$ and $Q$ on $M$ for which there is a horizontal curve joining them is defined as
$$ d_{CC}(P, Q) = \inf  \left \{ L(\gamma) \,; \, \gamma \mbox{ is a horizontal curve joining } P \mbox{ and } Q \right \}. $$

Keep in mind that in a sub-Riemannian manifold, it is not always certain that there is a straight path connecting two points. The Chow-Rashevskii connectivity theorem (see \cite[Theorem 3.5.2]{Ovidiu}) guarantees the existence of such path, but only if two extra conditions are met: the manifold $M$ is connected, and the distribution $\mathcal{H}$ satisfies the Hörmander condition (or bracket-generating condition).

The sub-Laplacian associated with the sub-Riemannian structure  is the natural generalization of the Laplace–Beltrami operator defined on a Riemannian manifold and is represented by
$$\Delta_M = \sum_{i=1}^k X_i^2 +Y,$$
where $X_1, \cdots, X_k$ are linearly independent real vector fields that are bracket generating (i.e. they together with all brackets span $T_pM$ for all $p \in M$), and $Y$ is any real vector field. 

If $M$ is a complete manifold in the sub-Riemannian metric associated with $\Delta_M$ and has a smooth volume denoted by $\mu$, the results from \cite{STRICHARTZ} demonstrate that $\Delta_M$ is essentially self-adjoint on the domain $C^{\infty}_c(M)$. Additionally, the heat semigroup $\{S_M(t)\}_{t \geq 0}$ associated with $\Delta_M$ is contractive on all $L^p(X, \mu)$ spaces, where $1 \leq p \leq \infty$. Furthermore, $S_M(t)$ is described by a positive $C^{\infty}(M)$ kernel and $S_M(t)f$ provides the unique solution to the Cauchy problem $u_t - \Delta_M u =0$, $u(0) =f$, under the assumption that $\|u( \cdot, t) \|_p$ is uniformly bounded for $1 < p < \infty$.

A basic example of a sub-Riemannian manifold is a typical Riemannian manifold whose distribution $\mathcal{H}$ matches the tangent bundle $TM$ (i.e., $k=N$). Examples of sub-Riemannian manifolds include $\mathbb{R}^N$ and the hyperbolic space $\mathrm{H}^N$, which are Riemannian manifolds. However, there are "pure" sub-Riemannian structures (i.e., sub-Riemannian manifolds that are not Riemannian manifolds), such as connected unimodular Lie groups like the Heisenberg group, B-groups, or K-type Groups, and more generally homogeneous Carnot groups (see \cite{BONFIGLIOLI}).

\subsection*{Carnot Groups}

The Lie group on $\mathbb{R}^N$,  $\mathbb{G} = (\mathbb{R}^N,\circ , \delta_{\lambda})$ is called a homogeneous Carnot group (or a homogeneous stratified group), of step $r$ and $N_1$ generators, if the following properties hold:
   \begin{enumerate}
       \item[(i)] $\mathbb{R}^N$ can be split as $\mathbb{R}^N = \mathbb{R}^{N_1} \times \cdots \times \mathbb{R}^{N_r} $, and the dilation $\delta_{\lambda} : \mathbb{R}^N \rightarrow \mathbb{R}^N$ given by
       $$ \delta_{\lambda}(x) = \delta_{\lambda}(x^{(1)},\cdots, x^{(r)}) =  (\lambda x^{(1)}, \lambda^2 x^{(2)} ,\cdots, \lambda^{r} x^{(r)}), \quad x^{(i)} \in \mathbb{R}^{N_i}  $$
is an automorphism of the group $\mathbb{G}$ for every $\lambda > 0$.
\item [(ii)]{(Hörmander condition)}  If $N_1$ is as above, let $Z_1 , \cdots , Z_{N_1}$ be the left invariant vector fields on $\mathbb{G}$ such that $Z_j(0) = \partial/ \partial x_j\lvert_0$ for $j = 1,\cdots,N_1$. Then 
$$\mbox{dim}( \mbox{span} \{Z_J(x)\, ; J = (j_1, \cdots, j_k) \in \{1, \cdots, N_1\}^k\,\}, \, k \in \mathbb{N} ) = N,$$
for every  $x\in \mathbb{R}^N$ where $Z_J$ is a commutator of length $k$ of $Z_1, \cdots ,Z_{N_1}$.
   \end{enumerate}
The vector fields $Z_1,...,Z_{N_1}$  are called the (Jacobian) generators of $\mathbb{G}$, whereas any basis for the $\mbox{span} \{Z_1,...,Z_{N_1}\} $ is called a system of generators of $\mathbb{G}$.

Homogeneous Carnot groups have a stratified Lie algebra \(\mathfrak{g}\) which can be decomposed into a direct sum of vector spaces: \(\mathfrak{g} = V_1 \oplus V_2 \oplus \cdots \oplus V_r\), where \([V_1, V_j] = V_{j+1}\) for \(1 \leq j < r\) and \([V_1, V_r] = \{0\}\), where $[\cdot , \cdot ]$ represents the Lie-bracket. The first layer, \(V_1 = \{Z_1,\cdots,Z_{N_1}\}\), generates the entire Lie algebra through the commutator bracket. The sub-Laplacian is the second-order differential operator constructed from vector fields that span the first layer \(V_1\), i.e.,
\begin{equation}\label{sub-laplaciano}
    \Delta_{\mathbb{G}}  = \sum_{j=1}^{N_1} Z_j^2.
\end{equation}
These sub-Laplacians are hypoelliptic, due to H\"ormander theorem \cite{Lars}. We will consider the homogeneous norm in $\mathbb{G}$ defined by 
$$
   \lvert \eta \rvert_{ \mathbb{G} } = \left( \sum_{j=1}^r \lvert x^{(j)} \rvert^{2/j}  \right)^{\frac{1}{2}}, \quad \eta = (x^{(1)}, \cdots, x^{(r)} ) \in \mathbb{G}.
$$
Let $d_{CC}( \eta , \xi)$ be the Carnot-Carathéodory distance in $\mathbb{G}$. Then $d_0(x) = d_{CC}(x,0)$ is a symmetric homogeneous norm on $\mathbb{G}$ and satisfies the pseudo-triangle inequality (see \cite[Proposition 5.1.8 and Theorem 5.2.8]{BONFIGLIOLI}). Since all homogeneous norms in $\mathbb{G}$ are equivalent (see \cite[Proposition 5.1.4]{BONFIGLIOLI}), the norm $\lvert \eta \rvert_{ \mathbb{G} }$ is equivalent to the Carnot-Carathéodory norm $d_0$. Additionally, the norm $\lvert \eta \rvert_{ \mathbb{G} } $ satisfies the pseudo-triangle inequality, which means that there exists a constant $k > 0$ such that: 
$$
    \lvert \eta \circ \xi \rvert_{ \mathbb{G} } \leq k \left(  \lvert \eta \rvert_{ \mathbb{G} } +  \lvert  \xi \rvert_{ \mathbb{G} }\right) \qquad \mbox{ for all }  \eta , \xi \in \mathbb{G}.
$$

The Lebesgue measure on $\mathbb{R}^N$ is the Haar measure for $\mathbb{G}$ (see \cite[Proposition 1.3.21]{BONFIGLIOLI}). We denote this measure by $\mu$. The Lebesgue measure is $\delta_{\lambda}$-homogeneous with respect to the dilations
$\{ \delta_{\lambda} \}_{\lambda>0}$. More precisely
\begin{equation}\label{homog-haar}
   \mu( \delta_{\lambda}(E) )  = \lambda^Q \mu (E),
\end{equation}
where $E \subset \mathbb{R}^N$ is a Lebesgue measurable set, and $Q = \sum_{j=1}^r j \cdot N_j.$ The number $Q$ is known as the homogeneous dimension of the homogeneous Carnot group $\mathbb{G}$.

Let \(\{S_{\mathbb{G}}(t)\}_{t \geq 0}\) be the symmetric sub-Markovian semigroup associated with the sublaplacian in \eqref{sub-laplaciano}. It is well known that the semigroup \(\{S_{\mathbb{G}}(t)\}_{t \geq 0}\) admits a right convolution kernel \(h_t\) such that 
$$ S_{\mathbb{G}}(t) u (\eta) = \int_{\mathbb{G}} u (\xi) h_t (\xi ^{-1} \circ \eta ) d\mu(\xi). $$
The kernel $h_t(\eta)$ is a smooth and positive function for $(t,\eta) \in (0, \infty) \times \mathbb{G}$, where $\| h_t \|_{L^1(\mathbb{G})} = 1$ for any $t>0$. This function satisfies the heat equation $\partial_t - \Delta_{\mathbb{G}} = 0$ (see \cite{SALOFF}). In addition, by combining the homogeneous structure (see \cite[Theorems IV. 4.2 and IV.4.3]{SALOFF}) and the polynomial growth of the homogeneous Carnot group $\mathbb{G}$ (given by \eqref{homog-haar}), we obtain the following Gaussian estimates for the heat kernel.
\begin{equation}\label{Carnotkernel}
	c_0t^{-\frac{Q}{2}} \exp \left(  - \dfrac{C_1 \lvert \eta \rvert_{\mathbb{G}}^2}{t} \right) \leq h_t(\eta) \leq C_0t^{-\frac{Q}{2}} \exp \left(  - \dfrac{c_1 \lvert \eta \rvert_{\mathbb{G}}^2}{t} \right) 
\end{equation}
for any $t > 0$ and $\eta \in \mathbb{G}$, where $c_0,C_0, c_1, C_1$ are positive constants. 


Let $ u_0 \in L_{+}^p(X, \mu)$ ($1 \leq p \leq \infty$), and $u_0 \neq 0$ on a ball of $\mathbb{G}$ (specifically on the ball $B_{\mathbb{G}}(\zeta, \delta)$, $\zeta \in \mathbb{G}, \delta>0$). Then, from the estimate \eqref{Carnotkernel} and Fatou’s lemma, we have
\begin{eqnarray*}
\liminf_{t \rightarrow \infty} t^{\frac{Q}{2}} S_{\mathbb{G}}(t)u_0(\eta) & \geq & \liminf_{t \rightarrow \infty} c_0 \int_{\mathbb{G}} \exp \left(  - \dfrac{C_1 \lvert \xi ^{-1} \circ \eta  \rvert_{\mathbb{G}}^2}{t} \right) u_0(\xi) \, d \mu(\xi) \\
& \geq &  \liminf_{t \rightarrow \infty} c_0 \int_{B_{\mathbb{G}}(\zeta, \delta)} \exp \left(  - \dfrac{C_1 \lvert \xi ^{-1} \circ \eta  \rvert_{\mathbb{G}}^2}{t} \right) u_0(\xi) \, d \mu(\xi) \\
& \geq & c_0\int_{B_{\mathbb{G}}(\zeta, \delta)} u_0(\xi) \, d \mu(\xi).   
\end{eqnarray*}
Thus, choosing $0 < C < c_0\int_{B_{\mathbb{G}}(\zeta, \delta)} u_0(\xi) d \mu(\xi)$, we obtain

  \begin{align}\label{lower-est-Heisenberg}
          S_{\mathbb{G}}(t)u_0 \geq C \, t^{-Q/2},
  \end{align}
 for $t$ sufficiently large. On the other hand, if $w_0 \in C_c(\mathbb{G})$, from the upper bound Gaussian estimate in \eqref{Carnotkernel}, we obtain
\begin{equation}\label{upper-est-heisenb}
    \|S_{\mathbb{G}}(t)w_0\|_{\infty } \leq C t^{-\frac{Q}{2}}
\end{equation}
for $t>0$ and some positive constant $C$. Here $C_c(\mathbb{G})$ is the set of continuous functions with compact support.

Based on the information above, we can apply Theorem \ref{thm1} in this context. We will begin by using Theorem \ref{thm1} to address problem \eqref{eq2.1} with a polynomial nonlinearity of the form 
$f_i(u) = u^{q_i},$
where $q_i>1$ for $i = 1,\cdots, k$.

In the work by Andrea Pascucci in \cite{Pascucci1}, the problem $u_t - \Delta_{\mathbb{G}} u = u^q$ was examined for the Carnot groups. It was shown that the critical exponent in this case, according to Fujita, is given by $q^* = 1+ 2/Q$, where $Q$ is the homogeneous dimension of $\mathbb{G}$. See also \cite{Ruzhansky1} and \cite{GEORGIEV} for other related advances.

The following result shows that for $f_i (u) = u^{q_i} \, (i=1,...,k)$ (or $f_i (u) = (1+u)[\ln(1+u)]^{q_i}$), the critical exponent is given by $$q^* = \max_{1 \leq i \leq k} \left\{ 1 + \frac{2+2 \alpha_i}{Q} \right\}. $$
\begin{cor}\label{coEXTRA} Let $X = \mathbb{G}$, $f_i (u) = u^{q_i} \, (q_i>1)$ (or $f_i (u) = (1+u)[\ln(1+u)]^{q_i}$), and $\ell_i(t) = t^{\alpha_i} \, (\alpha_i>-1)$, for $i = 1, \cdots, k.$ 
\begin{itemize}  
  \item[(i)] If \, $1 < q_i < 1 + {(2+2 \alpha_i)}/{Q} $ for some $i = 1,\cdots, k$. Then the problem $\eqref{eq2.1} $ has no nontrivial global solutions.
  \item[(ii)] If \, $q_i > 1 + {(2 +2 \alpha_i)}/{Q}$ for all $i = 1, \cdots, k$, then there is a nontrivial global solution to problem \eqref{eq2.1}.
\end{itemize}
\end{cor}
\begin{proof} {\bf Case} $f(u_i)=u_{i}^{q_i}$. (i) Let $u_0 \in L_+^p(X, \mu)$, where $u_0 \neq 0$, and let $K$ be a compact set in the topology of $\mathbb{G}$ such that $u_0(\xi) >0$ for a.e $\xi \in K$. By choosing $t \geq 1$ large enough so that the estimate \eqref{lower-est-Heisenberg} holds and 
$$ t^{Q(q_i-1)/2 -(1+\alpha_i)} \leq \dfrac{q_i-1}{(\alpha_i+1) C^{1-q_i}}, $$ we obtain
  \begin{align*}
      \begin{aligned}
          \int_{ \frac{1}{ \mu(K) } \int_K S_{\mathbb{G}} ( t ) u _0 d \mu(\xi)}^{\infty} \frac{d \sigma}{\sigma^{q_i}} &=   \dfrac{ \left[ \frac{1}{ \mu(K) } \int_K S_{\mathbb{G}} ( t ) u _0 d\mu(\xi) \right]^{1-q_i}}{q_i-1} \\
          &\leq \dfrac{C^{1-q_i}}{q_i-1} t^{\frac{Q(q_i-1)}{2}}  ~ \leq ~ \int_{0}^{t} \sigma^{\alpha_i} ~ d \sigma.
      \end{aligned}
  \end{align*}
 Therefore, Theorem \ref{thm1} (ii) implies that the solution of problem (\ref{eq2.1}) with initial condition $u_0$ is not global. 

(ii) Let $w_0  \in C_c(\mathbb{G}) \subset L_+^p(X, \mu)$, $w_0 \neq 0$, which will be conveniently chosen later. The estimate \eqref{upper-est-heisenb} gives us
\begin{align*}
\sum_{i=1}^{k}    \int_{0}^{\infty}  \ell_i(\sigma)  F_i(\|S_{\mathbb{G}}(\sigma)w_{0}  \|_{\infty})  d\sigma  & =\sum_{i=1}^{k}    \int_{0}^{\infty}  \sigma^{\alpha_i} \|  S_{\mathbb{G}}(\sigma)w_{0}  \|_{\infty}^{q_i-1}  d\sigma  \\
&=  \sum_{i=1}^{k} \left (\int_0^{m} + \int_{m}^{\infty} \right )    \sigma^{\alpha_i} \|  S_{\mathbb{G}}(\sigma)w_{0}  \|_{\infty}^{q_i-1}  d\sigma\\
    \noalign{\medskip}
    & \leq   \sum_{i=1}^{k}  \int_0^{m} \|w_0 \|_{\infty}^{q_i - 1}  {\sigma}^{\alpha_i} d\sigma \\
    & \, \, \,  +  \sum_{i=1}^{k}  \int_m^{\infty} {\sigma}^{-{Q(q_i-1)}/{2} + \alpha_i} d\sigma 
\end{align*}
The convergence of the last integral is ensured because $q_i > 1+ {(2+2\alpha_i )}/{Q}$. Therefore, the condition (\ref{GLOB1}) of Theorem \ref{thm1} (i) is verified by choosing $m$ sufficiently large and $\|w_0 \|_{\infty}$ sufficiently small.\\

\noindent
{\bf Case }$f_i (u) = (1+u)[\ln(1+u)]^{q_i}$. (i) Let $t >1,$ which will be choose large enough later and $ u_0  \in L_+^p(X, \mu)$, $u_0 \neq 0$. Utilizing the same arguments as the previous case and since $1 < q_i < 1 + (2+2 \alpha_i)/Q$, we have
 \begin{align*}
      \begin{aligned}
        & \left[ \int_{ \frac{1}{\mu(K)} \int_K S_{\mathbb{G}} ( t ) u _0 d \mu}^{\infty} \frac{d \sigma}{ (1+\sigma)[\ln(1+\sigma)]^{q_i}      } \right]^{-1}  \int_{0}^{t} \sigma^{\alpha_i} ~ d \sigma\\
         &\geq \frac{q_i -1}{\alpha_i +1}   \left[ \ln \left( 1+ 
 \frac{1}{\mu(K)} \int_K S_{\mathbb{G}} ( t ) u _0 d \mu \right) \right]^{q_i -1} t^{\alpha_i+1}\\
          &\geq \frac{q_i -1}{\alpha_i +1} \left[ \ln \left( 1+ C t^{-Q/2} \right) \right]^{q_i - 1} t^{\alpha_i+1} \\
          & =  \frac{q_i -1}{\alpha_i +1} \left[ \ln \left( 1+ C t^{-Q/2} \right)  t^{\frac{  \alpha_i+1}{q_i - 1}} \right]^{q_i - 1}\\
          & \geq \frac{q_i -1}{\alpha_i +1} \left[ \ln \left( 1+ C t^{-Q/2} \right) t^{\frac{  \alpha_i+1}{q_i - 1}}   \right]^{q_i - 1} \\
          & \geq \frac{q_i -1}{\alpha_i +1} \left[ C t^{ - \frac{Q}{2}} t^{\frac{  \alpha_i+1}{q_i - 1}}   \right]^{q_i - 1} >1
      \end{aligned}
  \end{align*}
for $t>1$ large enough. This is possible because $\ln (1+x)>x/(1+x)$ for $x>-1$.


(ii) Let $w_0  \in C_c(\mathbb{G}) \subset L^p_+(X, \mu)$, $w_0 \neq 0$, which will be conveniently chosen later. From \eqref{upper-est-heisenb} and using the following known upper bound estimate $\ln(x+1) \leq x \ (x>-1)$, we have
\begin{align*}
&\sum_{i=1}^{k}    \int_{0}^{\infty}  \ell_i(\sigma)  F_i(\|S_{\mathbb{G}}(\sigma)w_{0}  \|_{\infty})  d\sigma \\
  &= \displaystyle \sum_{i=1}^k \int_{0}^{\infty} \sigma^{\alpha_i} \dfrac{ \left( 1+ \|S_{\mathbb{G}}(\sigma)w_0\|_{\infty} \right) \left[ \ln \left( 1+ \|S_{\mathbb{G}}(\sigma)w_0\|_{\infty} \right) \right]^{q_i}}{\|S_{\mathbb{G}}(\sigma)w_0\|_{\infty}} d\sigma \\
  &\leq  \left( 1+ \|w_0\|_{\infty} \right)\displaystyle \sum_{i=1}^k \int_{0}^{\infty} \sigma^{\alpha_i}   \|S_{\mathbb{G}}(\sigma)w_0\|_{\infty}^{q_i-1} d\sigma \\
  & \leq \left( 1+ \|w_0\|_{\infty} \right)\displaystyle \sum_{i=1}^k  \left (\int_0^{m} + \int_{m}^{\infty} \right )    \sigma^{\alpha_i} \|  S_{\mathbb{G}}(\sigma)w_{0}  \|_{\infty}^{q_i-1}  d\sigma\\
    \noalign{\medskip}
    & \leq\left( 1+ \|w_0\|_{\infty} \right) \left[ \sum_{i=1}^{k}  \int_0^{m} \|w_0 \|_{\infty}^{q_i - 1}  {\sigma}^{\alpha_i} d\sigma + \displaystyle \sum_{i=1}^k \int_m^{\infty} {\sigma}^{-{Q(q_i-1)}/{2} + \alpha_i} d\sigma \right].
\end{align*}
The convergence of the last integrals is guaranteed because $-{Q(q_i-1)}/{2} + \alpha_i<-1$ for all $i=1, \cdots, k$. Therefore, the result follows from Theorem \ref{thm1} (i) since condition (\ref{GLOB1}) is satisfied by choosing $m$ sufficiently large and  $\|w_0 \|_{\infty}$ sufficiently small.
\end{proof}

We will now present a result that involves exponential nonlinearities
$$f_i(u) =  e^{ \beta_i u} -1, \qquad \beta_i>0, \mbox{ for } i=1,\cdots, k.$$
These nonlinearities are associated with the Frank-Kamenetsky equation $u_t - \Delta u = e^{u},$ which is found in models of combustion theory and has been extensively studied in various works (see \cite{FUJISHIMA}, \cite{IOKU}, \cite{TELLO}, and the references therein).
\begin{cor}\label{coEXP} Let $X = \mathbb{G}$ and $u_0 \in L^p_+(X, \mu)$, $u_0 \neq 0$. If   $f_i (u) = e^{ \beta_i u} -1 \, \, (\beta_i>0)$, and $\ell_i(t) = t^{\alpha_i} \, \, (\alpha_i>-1)$, for $i = 1, \cdots, k.$ Then the problem \eqref{eq2.1} has no nontrivial global solutions.
\end{cor}
\begin{proof}
    Note that the functions $f_i (u)$ are convex for $i = 1,\cdots, k$ and $f_i(0)=0$. Now, we will make a similar argument to the previous corollaries. We have
 \begin{align*}
      \begin{aligned}
          \int_{ \frac{1}{\mu(K)}\int_K S_{\mathbb{G}} ( t ) u _0 d \mu}^{\infty} \frac{d \sigma}{ e^{ \beta_i \sigma } -1      }  &=   \dfrac{ 1}{\beta_i}  \ln \left( \dfrac{1}{ 1 - \exp\left( -  \frac{\beta_i}{\mu(K)} \int_K S_{\mathbb{G}} ( t ) u _0 d \mu \right) } \right) \\
          &\leq \dfrac{ 1}{\beta_i}     \ln \left( \dfrac{1}{ 1 - \exp\left( - c_1 \beta_i  t^{-Q/2} \right) } \right)   \\
      \end{aligned}
  \end{align*}
On the other hand, for $t \geq 1$ large enough, we have $$ t^{-(\alpha_i+1)}  \ln \left( \dfrac{1}{ 1 - \exp\left( - c_1 \beta_i  t^{-Q/2} \right) } \right) \leq  \dfrac{\beta_i}{\alpha_i+1}.$$ Hence, we can conclude
$$ \int_{ \frac{1}{\mu(K)} \int_K S_{\mathbb{G}} ( t ) u _0 d \mu}^{\infty} \frac{d \sigma}{ e^{ \beta_i \sigma } -1      } \leq ~ \int_{0}^{t} \sigma^{\alpha_i} ~ d \sigma$$
for $t$ sufficiently large. Thus, the result follows from Theorem \ref{thm1} (ii).
\end{proof}

\subsection*{The hyperbolic space $\bold{H}^N$}

Several models describe the hyperbolic space $\bold{H}^N$. One way is to visualize $\bold{H}^N$ as an embedded hyperboloid in $\mathbb{R}^{N+1}$. Another approach is to use one of the two Poincaré realizations, where we can identify $\bold{H}^N$ with the unit ball in $\mathbb{R}^{N}$ or with the upper half-space, each having an appropriate metric. For a more detailed explanation of hyperbolic space, see \cite{BENEDETTI} and \cite{RATCLIFFE}. In our context, the hyperbolic space $\bold{H}^N$ refers to the unit ball $B_1$ of $\mathbb{R}^{N}$ centered at the origin, equipped with the Riemannian measure, and the Poincaré metric
    $$ ds = \dfrac{2}{1- \lvert x \rvert^2} \lvert dx \rvert .$$
 The Poincaré metric on the unit ball leads to distances expanding as one approaches the edge of the ball, reflecting the infinite nature of hyperbolic space.  The Laplace–Beltrami operator on $\bold{H}^N$ in geodesic coordinates it reads as
  $$ \Delta_{\bold{H}^N} = \dfrac{1}{ (\sinh \rho)^{N-1} } \dfrac{\partial}{\partial \rho } \left(   (\sinh \rho)^{N-1} \dfrac{\partial u }{\partial \rho}
 \right)     + \dfrac{1}{(\sinh \rho)^2 } \Delta_{\theta} u,$$ 
where $\Delta_{\theta}$ denotes the Laplace–Beltrami operator on $\mathbb{S}^{N-1}$ and $ d\mu = (\sinh \rho)^{N-1} d \rho d \theta$, with  $\theta \in \mathbb{S}^{N-1}$, $\rho \in [0,\infty)$ denoting the geodesic polar coordinates.

From Davies (see \cite[Section 5.7]{DAVIES}) the heat kernel $K_N(x, y, t)$ in $\bold{H}^N$ satisfies $\int_{\mathrm{H}^N} K_N(x, y, t) \, d\mu(y) = 1$.
This implies that the associated convolution semigroup $\{S_{\bold{H}^N}(t)\}_{t\geq 0}$ is sub-Markovian. Furthermore, $K_N(x, y, t)$ satisfies the two-sided estimates of the heat kernel:
$$
c_N^{-1} h_N(d(x, y), t) \leq K_N(x, y, t) \leq c_N h_N(d(x, y), t),
$$
 where $c_N$ is a positive constant, $d(x, y)$ is the hyperbolic distance between two points $x$ and $y$ in $\bold{H}^N$, and
$$
\begin{aligned}
h_N(d(x, y), t) &= (4 \pi t)^{- \frac{N}{2}} (1 + d(x, y))(1 + d(x, y) + t)^{\frac{N-3}{2}} \\
&\quad \times e^{-\lambda_0 t - \frac{N-1}{2}d(x, y) - \frac{d(x, y)^2}{4t}}
\end{aligned}
$$
with $\lambda_0 = (N-1)^2/4$. In particular, it follows from \cite[Corollary of Theorem 3.1]{MANDOUVALOS} or \cite[Proposition 2.3]{VÁZQUEZ}, that for $u_0 \in L^{\infty}( \bold{H}^N ) \cap L ^{1}( \bold{H}^N )$, there exists $C>0$ such that
\begin{equation}\label{estimativa-Vasquez}
\| S_{\bold{H}^N}(t) u_0 \|_{\infty} \leq Ct^{-3/2} e^{- \lambda_0 t} \mbox{ for } t \geq 1.    
\end{equation}

On the other hand, if $u_0 \geq 0$ and $u_0 \neq 0$ in a ball of $\bold{H}^N$, then for any $t > 0$ there exists a function $f \in C( \bold{H}^N )$, $f>0$ such that 
\begin{equation}\label{Bandle}
[S_{\bold{H}^N}(t)u_0](x) \geq f(x) t^{-3/2} e^{-\lambda_0 t }
\end{equation}
for any $t \in [\varepsilon, \infty)$ and $x\in \bold{H}^N$(see \cite[Lemma 5.1]{BANDLE}).


The following corollary shows that in $\bold{H}^N$, the problem \eqref{eq2.1} with $f_i (u) = u^{q_i}$ or $f_i (u) = (1+ u)[\ln(1+u)]^{q_i}$, where $q_i>1$ and $\ell_i(t) = e^{\alpha_i t}$ ($\alpha_i>0$), for $i = 1,\cdots, k$, has a critical exponent given by
$$q^* = \max_{1 \leq i \leq k} \left\{ 1 + \frac{4 \alpha_i}{(N-1)^2} \right\}. $$ 
This generalizes the results of C. Bandle et al. in \cite{BANDLE} and Hui Wu in \cite{HUI}.

\begin{cor}\label{co2} Assume that $X = \bold{H}^N$ is the $N$-dimensional hyperbolic space, $f_i (u) = u^{q_i}$ or $f_i (u) = (1+u)[\ln(1+u)]^{q_i}$, $q_i>1$ and $\ell_i(t) = e^{\alpha_i t} \, \,(\alpha_i>0)$, for $i = 1,\cdots, k.$ 
\begin{itemize}
  \item[(i)] If \, $1 < q_i < 1 + {4 \alpha_i}/{(N-1)^2} $ for some $i = 1, \cdots, k$, then the problem $\eqref{eq2.1} $ has no nontrivial global solutions.
  \item[(ii)] If $q_i > 1 + {4 \alpha_i}/{(N-1)^2}$ for all $i = 1,\cdots, k$, then problem \eqref{eq2.1} has a nontrivial global solution.
\end{itemize}
\end{cor}
\begin{proof} We only treat the polynomial case, because the logarithmic case can be handled similarly arguing as the proof of Corollary \ref{coEXTRA}.

(i) Let $u_0 \in L^p_{+}(X,\mu)$, $u_0 \neq 0$ and $K \subset X$ be a compact subset such that the restricted function $u_0\mid_{K}$ is nonzero. From estimate \eqref{Bandle} we have
  \begin{align*}
      \begin{aligned}
         \int_{ \frac{1}{\mu(K)}\int_K S_{\bold{H}^N}(t) u _0 d \mu}^{\infty} \frac{d \sigma}{\sigma^{q_i}}  &=   \dfrac{   \left[ \frac{1}{\mu(K)} \int_K S_{\bold{H}^N}(t) u _0 d \mu \right]^{1-q_i}}{q_i-1}\\
         &\leq \dfrac{\left[ \frac{1}{\mu(K)} \int_K f d \mu \right]^{1-q_i}}{q_i-1} t^{3(q_i-1)/2} e^{\lambda_0 (q_i-1) t } \\ 
      \end{aligned}
  \end{align*}
Note that the value $[{\mu(K)}]^{-1} \int_K f d \mu $ is finite, because $f \in C(\bold{H}^N)$. Hence, recalling that $\lambda_0 = (N-1)^2/4$ and $\lambda_0 (q_i-1)<\alpha_i$, we have
$$ \dfrac{t^{3(q_i-1)/2} e^{\lambda_0 (q_i-1) t } }{ e^{\alpha_i t} -1 } \leq \dfrac{q_i-1} {\alpha_i \left[ \frac{1}{\mu(K)} \int_K f d \mu \right]^{1-q_i} }, $$
for $t > \varepsilon>0$ sufficiently large. Then
 \begin{align*}
      \begin{aligned}
         \int_{ \frac{1}{\mu(K)}\int_K S_{\bold{H}^N}(t) u _0 d \mu}^{\infty} \frac{d \sigma}{\sigma^{q_i}}  & \leq \dfrac{e^{\alpha_i t} -1}{\alpha_i}\\
        & =   \int_{0}^{t} \ell_i(\sigma) d \sigma.
      \end{aligned}
  \end{align*}
Therefore, the result follows from Theorem \ref{thm1} (ii).\\

(ii) Let $w_0 \in C_c(X) \subset L^p_{+}(X,\mu)$, which will be chosen to be sufficiently small. To prove the existence of a global solution for problem \eqref{eq2.1}, we only need to show that
\begin{equation}\label{glob-hiper1}
 \sum_{i=1}^{k}    \int_{0}^{\infty}  \ell_i(\sigma)  F_i(\|S_{\mathbb{G}}(\sigma)w_{0}  \|_{\infty})  d\sigma =   \displaystyle \sum_{i=1}^k \int_{0}^{\infty} e^{ \alpha_i \sigma} \|S_{\bold{H}^N}(\sigma)w_0\|_{\infty}^{q_i-1} d\sigma   < 1,
\end{equation}
In fact, using  \eqref{estimativa-Vasquez}, we have
\begin{align*}
   & \displaystyle \sum_{i=1}^k  \displaystyle \int_{0}^{\infty} e^{ \alpha_i \sigma}  \|S_{\bold{H}^N}(\sigma)w_0\|_{\infty}^{q_i-1} d\sigma \\
   & = \displaystyle \sum_{i=1}^k  \displaystyle \int_{0}^{t_0} e^{ \alpha_i \sigma}  \|S_{\bold{H}^N}(\sigma)w_0\|_{\infty}^{q_i-1} d\sigma + \displaystyle \sum_{i=1}^k 
  \displaystyle \int_{t_0}^{\infty} e^{ \alpha_i \sigma}  \|S_{\bold{H}^N}(\sigma)w_0 \|_{\infty}^{q_i-1} d\sigma\\
   & \leq \displaystyle \sum_{i=1}^k  \|w_0\|_{\infty}^{q_i-1} \int_{0}^{t_0} e^{ \alpha_i \sigma} d\sigma   + C \displaystyle \sum_{i=1}^k \displaystyle \int_{t_0}^{\infty}  \sigma^{-3(q_i-1)/2} e^{ -[  \lambda_0 (q_i-1) - \alpha_i] \sigma} d\sigma .
\end{align*}
for $t_0 \geq 1$. Since $\lambda_0 (q-1) - \alpha_i>0$ for $i=1,...,k$ and based on the properties of the incomplete gamma function, we can conclude that the above integrals are bounded. Therefore, the estimate of \eqref{glob-hiper1} holds true if $t_0$ is sufficiently large and $\|w_0\|_{\infty}$ is sufficiently small.  
\end{proof}

The exponential case of the nonlinearity is given in the following result.

\begin{cor}\label{coEXP-hiperbolico}  Let $X = \bold{H}^N$ be the $N$-dimensional hyperbolic space and $u_0 \in L^p_+(X, \mu)$, $u_0 \neq 0$. If   $f_i (u) = e^{ \beta_i u} -1 \, \, (\beta_i>0)$, and $\ell_i(t) = e^{\alpha_i t} \, \,(\alpha_i>0)$ or $\ell_i(t) = t^{\alpha_i} \, (\alpha_i>-1)$, for $i = 1, \cdots, k.$, then the problem \eqref{eq2.1} has no nontrivial global solutions.
\end{cor} 
\begin{proof} For the case $\ell_i(t) = e^{\alpha_i t}$, according to the proofs of Corollaries \eqref{coEXP} and \eqref{co2} (i), it is enough to observe that 
 $$\dfrac{1 }{ e^{\alpha_i t} -1 }  \ln \left( \dfrac{1}{ 1 - \exp\left( - c_1 \beta_i  t^{-3/2} e^{-\lambda_0 t } \right) } \right) \leq  \dfrac{\beta_i} {\alpha_i },$$ 
 for $t \geq 1$ large enough, where $c_1 = [{\mu(K)}]^{-1} \int_K f d \mu $ is finite. Hence, we can conclude
$$ \int_{ \frac{1}{\mu(K)} \int_K S_{\bold{H}^N} ( t ) u _0 d \mu}^{\infty} \frac{d \sigma}{ e^{ \beta_i \sigma } -1      } \leq ~ \int_{0}^{t} e^{\alpha_i \sigma} ~ d \sigma$$
for $t$ sufficiently large. Thus, the result follows from Theorem \ref{thm1} (ii). 

The case $\ell_i(t) = t^{\alpha_i}$ is showed similarly.
\end{proof}

\subsection{Applications on graphs}

We will begin with a brief explanation of Laplacians, semigroups, and heat kernels on graphs. Further details can be found in \cite{KELLER1}, \cite{Keller}, \cite{HAESELER}, \cite{Horn}, and \cite{LIU}, as well as the references therein.

To present our findings, we will review certain definitions related to graphs.
Let $G = (V, E)$ be a graph, where \( V \) denotes the set of vertices, and \( E \) denotes the set of edges. A domain, denoted by \( \Omega \), is defined as a connected subset of the vertices \( V \). In this section, it is assumed that the graph \( G \) fulfills the subsequent conditions:
\begin{itemize}
    \item[(1)](Simple) The graph $G$ does not contain loops or multiple edges.
    \item[(2)] (Locally finite) For any vertex $x \in V$, there are only a finite number of vertices $y \in V$ such that $xy \in E$, where $xy$ denotes an edge of the graph with $x$ and $y$ as its endpoints.
    \item[(3)] (Connected) For any \( x, y \in V \), there are a finite number of edges connecting \( x \) and \( y \).
    \item[(4)] (Symmetric) For any $x, y \in V$, let $\omega: V \times V \rightarrow \mathbb{R}$ be a positive symmetric weight, that is, $\omega_{xy} > 0$ and $\omega_{xy} = \omega_{yx}$. Additionally, we assume that $\displaystyle \omega_0 = \inf_{xy \in E} \omega_{xy} > 0$.
    \item[(5)] (Positive finite measure) The function $\mu: V \rightarrow \mathbb{R}^+$ defines a positive finite measure on the graph $G$ and satisfies $\displaystyle \mu_0 = \inf_{x \in V} \mu(x) > 0$.
    \item[(6)] ($D_{\omega} < +\infty, D_\mu < +\infty$) Let $ \displaystyle \mu_1 = \sup_{x \in V} \mu(x)$ and $ \displaystyle m(x) =
\sum_{y \sim x} \omega_{xy}$,
where $y \sim x$ means $xy \in E$. $D_{\omega}$ and $D_{\mu}$ are defined by
$$D_\omega = \frac{\mu_1}{\omega_0} \, \,
\mbox{ and } \, \,
D_\mu = \sup_{x \in V} \frac{m(x)}{\mu(x)},$$
we assume $D_\omega < +\infty$, $D_{\mu} < +\infty$.
\end{itemize}
Let $C(V)$ denote the set of real functions on $V$, and let $C_c(V)$ represent the set of real functions on $V$ that have finite support.

The graph $G$ is said to satisfy the \textit{curvature-dimension condition} $CDE'(n,0)$, where $n > 0$, if for any $x \in V$ and any function $f: V \to \mathbb{R}^+$, the following inequality holds:
\[
\Gamma_2(f)(x) - \Gamma \left( f,  \dfrac{\Gamma (f)}{f} \right) \geq \frac{1}{n} ( f(x))^2 [(\Delta_{G} \log f)(x)]^2 ,
\]
where $\Delta_{G}$ is the $\mu$-Laplacian on $G$, defined by
\[
\Delta_{G} f(x) = \frac{1}{\mu(x)} \sum_{y \sim x} \omega_{xy} \big( f(y) - f(x) \big),
\]
where $\omega(x,y)=\omega_{xy}$ and $y\sim x$ means $xy\in E$; $xy$ is an edge of the graph and the vertices $x$, $y$ are called the endpoints of the edge. The \textit{carré du champ} operator $\Gamma(f)$ is defined by
\[
\Gamma(f)(x) = \frac{1}{2\mu(x)} \sum_{y \sim x} \omega_{xy}  \big( f(y) - f(x) \big)^2,
\]
and the iterated carré du champ operator $\Gamma_2(f)$ is defined by
\[
\Gamma_2(f)(x) = \frac{1}{2\mu(x)} \sum_{y \sim x} \omega_{xy}  \big( \Gamma(f)(y) - \Gamma(f)(x) \big).
\]

The distance $d(x, y)$ between $x$ and $y$ is defined by the number of edges belonging to the shortest path $\gamma$ connecting $x$ and $y$. Then we can define balls centered at $x\in V$ with radius $r>0$ by $B_r(x) = \{y \in V : d(x, y) \leq  r\}.$ The volume of the ball $B_r(x)$ is denoted as $V_{\mu}(x,r)$ and is defined by $V_\mu(x,r) = \sum_{y \in B_r(x)} \mu (y)$.


According to the preceding sections, the bounds for the heat kernel are essential. For this reason, we will utilize volume growth conditions, such as those undertaken by Lin and Wu (\cite{Lin1}, \cite{Wu1}) and Liu (\cite{LIU}). Indeed, Lin and Wu obtained Fujita-type results by assuming the following volume growth conditions: There exist $c > 0$ and $x_0 \in V$ with $u_0(x_0) > 0$ such that the inequality
\begin{equation} \label{GROWTH1}
 c^{-1}r^d \leq V_\mu(x_0, r ) \leq c \ r^d \, (d > 0)
\end{equation}
holds for all large enough $r \geq r_0 > 0$. This condition is known as uniform polynomial volume growth of degree $d$ and was recently generalized in \cite{LIU}. 

We assume the existence of a vertex \( x_0 \in V \) satisfying \( u_0(x_0) > 0 \), given that \( u_0(x) \) represents a nontrivial initial value. We then fix \( x_0 \) and define \( B_r = B(x_0, r) \) for any positive integer \( r \). Furthermore, we assume that \( G \) conforms to the general volume growth condition. Specifically, there exist constants \( c' > 0 \) and \( c'' > 0 \), and a vertex \( x_0 \in V \) with \( u_0(x_0) > 0 \), which guarantee that for all \( r \geq r_0 > 0 \), the following inequality holds true:
\begin{equation} \label{growth-grap}
c' r^d \ln^{-\xi }(1 + r) \leq V_\mu (x_0, r) \leq c'' r^d \ln^{\eta} (1 + r), 
\end{equation}
where \( d > 0 \), and \( \xi, \eta \geq 0 \). This condition was utilized in \cite{LIU}. Note that when $\xi=\eta=0$, condition (\ref{growth-grap}) is equivalent to condition (\ref{GROWTH1}).

The heat kernel $p_G(t,x,y)$ on $G$ is the smallest positive fundamental solution to the heat equation ${\partial u}/{\partial t} = \Delta_G \, u.$

We now review some well-known properties of the heat kernel $p_{G}(t, x, y)$ as follows (see \cite{Horn}, \cite{Wu1}).
\begin{proposition} \label{KernProp} For any $t, s > 0$, and $x, y$ belonging to $V$, we have
\begin{itemize}
    \item[(i)] $p_G(t, x, y) = p_G(t, y, x) >0$.
    \item[(ii)] $ \displaystyle \sum_{y \in V} \mu(y) p_G(t, x, y) \leq 1$.
    \item[(iii)] $ \displaystyle \sum_{z \in V} \mu(z) p_G(t, x, z) p_G(s, z, y) =  p_G(t+s, x, y) $.
\end{itemize}  
\end{proposition}

Under the curvature condition $CDE'(n,0)$, Horn et al. in \cite{Horn} proved the following Gaussian-type upper and lower bounds.
\begin{prop}\label{prop-esti-ker} Suppose that 
$ D_{\omega} < \infty, \, D_{\mu} < \infty$ 
and $G$ satisfies $CDE'(n,0)$. Then there exists a positive constant $C_1$ such that, for any $x, y \in V$ and $t>0$,
$$p_G(t,x,y) \leq \dfrac{C_1}{V_\mu(x, \sqrt{t})} $$
where $C_1$ depends on $n, D_{\omega
}, D_{\mu}$. Furthermore, for any $t_0 > 0$, there exist positive constants $C_2$ and $C_3$ such that
$$ p_G(t,x,y) \geq \dfrac{C_2 e^{-C_3d^2(x,y)/t}}{V_\mu(x, \sqrt{t})}$$
for all $x, y\in V$ and $t>t_0$, where $C_2$ depends on $n$ and $C_3$ depends on $n, D_{\omega
}, D_{\mu}$. In particular, for any $t_0 > 0$, we have
$$p_G(t,x,y) \geq \dfrac{C_2}{V_\mu(x, \sqrt{t})} $$
for all $x\in V$ and $t>t_0$.
\end{prop}
  
Haeseler et al. in \cite{HAESELER} has proven that $D_\mu < \infty$ is equivalent to the $\mu$-Laplacian $\Delta_G$ be bounded on $L^p(V)$ for all $1 \leq p \leq \infty$, where
the norm $\|\cdot\|_{L^p(V)}$ of $L^p(V)$ is defined as
$$
 \|f\|_{L^p(V)} = \left\{\begin{array}{cr}
 \left( \displaystyle \sum_{x \in V } \mu(x) \lvert f(x) \rvert ^p \right)^{1/p} \,\,\, , & \quad1 \leq p < \infty\\[10pt]
  \displaystyle \sup_{x \in V } \lvert f(x) \rvert \,\,\, ,  & \quad p = \infty.
      \end{array}\right.
$$
So we can define the heat semigroup $\{ S_{G}(t) \}_{t\geq 0}$ of $\Delta$ by
$$ [S_{G}(t)f](x) = \sum_{y \in V} \mu(y) p_{G}(t,x,y) f(y) . $$ 
Furthermore by \cite{Keller}, if $D_\mu < \infty$ then the graph $G$ is stochastically complete, i.e., $\sum_{y \in V} \mu(y) p_{G}(t,x,y) =1.$ for any $t > 0$ and $x \in V$. This implies that the semigroup $\{ S_{G}(t) \}_{t\geq 0}$ is sub-Markovian (see also \cite[Proposition 3.1]{Horn}). Thus, Theorem \ref{AbstractTHEO} can be applied in the following form.


\begin{cor}\label{co-graf1}Let $G = (V, E)$ be a graph that satisfies conditions $(1)$-$(6)$. Assume that $G$ satisfies the $CDE'(n,0)$ condition and complies with the volume growth condition as described in equation \eqref{growth-grap}.
\begin{itemize}
  \item[(i)] Let $f_i (u) = u^{q_i}$ or $f_i (u) = (1+u)[\ln(1+u)]^{q_i}$ with $q_i>1$, and $\ell_i(t) = t^{\alpha_i} \, (\alpha_i>-1)$, for $i = 1, \cdots, k$.
  \begin{itemize}
  \item[(a)] If \, $1 < q_i < 1 + {(2+2 \alpha_i)}/{d} $ for some $i = 1,\cdots,k$, then the problem $\eqref{eq2.1} $ has no nontrivial global solutions. 
  \item[(b)]If $q_i > 1 + {(2 + 2 \alpha_i)}/{d}$ for all $i = 1,\cdots, k$, then there is initial data small enough such that the problem \eqref{eq2.1} has a nontrivial global solution.
  \end{itemize}
  \item[(ii)] If   $f_i (u) = e^{ \beta_i u} -1$, $\beta_i>0$, $\ell_i(t) = t^{\alpha_i}$, $\alpha_i>-1$ for $i = 1,\cdots,k$ and $u_0 \geq 0$, $u_0 \neq 0$, then the problem \eqref{eq2.1} has no nontrivial global solutions.
\end{itemize} 
\end{cor}

\begin{proof} We only show (i) for the polynomial case, because the logarithmic case and the proof of (ii)  can be obtained by adapting the proofs of Corollaries \ref{coEXTRA} and \ref{coEXP}. 

(a) Let $u_0 \in L^p_{+}(V)$, $u_0 \neq 0$, and $x_0 \in V$ such that $u_0(x_0)>0$. Considering the Gaussian heat kernel estimate from Proposition \ref{prop-esti-ker} and the volume growth condition $V_\mu(x_0,\sqrt{t}) \leq c'' t^{d/2} [\ln (1+ \sqrt{t})]^{\eta}$ for $t \geq r_0^2$ as stated in \eqref{growth-grap}, we deduce 
    \begin{align*}
        \begin{aligned}
            [S_{G}(t)u_0](x) &= \sum_{y \in V} \mu(y) p_G(t,x,y) u_0(y)\\
            & \geq \mu(x_0) p_G(t,x,x_0) u_0(x_0)\\
            & \geq \dfrac{C_2 \mu(x_0) u_0(x_0) }{V_\mu(x_0, \sqrt{t})} \\
            & \geq C t^{-d/2} [\ln (1+ \sqrt{t})]^{- \eta},
        \end{aligned}
    \end{align*}
    for any $t > \max\{ t_0, r_0^2\}$ and for any $x \in V$. Thus, we arrive at 
    \begin{align*}
      \begin{aligned}
         \int_{ \frac{1}{V_\mu(x_0, r)} \sum_{x \in B_r(x_0)} \mu (x)[S_G ( t ) u _0](x) }^{\infty} \frac{d \sigma}{\sigma^{q_j}}  &\leq   \dfrac{   \left[  C t^{-d/2} [\ln (1+ \sqrt{t})]^{- \eta}  \right]^{1-q_j}}{q_j-1}\\
         &\leq  \dfrac{ C^{1-q_j} }{q_j-1} t^{d(q_j-1)/2} [\ln (1+ \sqrt{t})]^{\eta(q_j-1)}  \\ 
         & \leq \dfrac{t^{\alpha_j +1} }{\alpha_j+1}\\
         & =   \int_{0}^{t} h_j(\sigma) d \sigma,
      \end{aligned}
  \end{align*}
for $t$ large enough. The penultimate inequality holds provided that we select $t \geq \max\{ t_0, r_0^2\}$ to be sufficiently large, such that 
$$ t^{d(q_j-1)/2 - (\alpha_j+1)} [\ln (1+ \sqrt{t})]^{\eta(q_j-1)} \leq \frac{q_j -1}{(\alpha_j+1) C^{1-q_j}}. $$
This condition is achievable as long as $q_j < 1 + {2(\alpha_j+1)}/{d}.$ Therefore, the result is established by Theorem  \ref{thm1} (ii). \\


(b) Let $u_0 \in C(V)$ and $\lambda > 0$ such that $0 \leq u_0(x) \leq \lambda \, p_G(\varepsilon_0, x_0, x)$, where $u_0$ is nonzero and $u_0(x_0) > 0$, with constants $r_0^2 > \varepsilon_0 > 0$. The constant $\lambda > 0$ will be chosen later to be sufficiently small. It is important to note that, as per Proposition \ref{prop-esti-ker}, for every $n \in \mathbb{N}$, one can select $\lambda >0$ such that $\lambda \, p_G(\varepsilon_0, x_0, x)\leq {1}/{n}$. Consequently, it can be deduced that \( u_0 \in L^\infty(V) \cap L^p(V) \) for any \( p \geq 1 \), and  
$$\|  S_G(t)u_{0} \|_{L^{\infty}(V)} \leq \frac{1}{n},$$
for all $t>0$.
Utilizing heat kernel estimates (see Proposition \ref{KernProp}) and considering the volume growth condition \( c' (t+\varepsilon_0)^{d/2} [\ln (1+\sqrt{t+\varepsilon_0})]^{-\xi}\leq V_\mu(x_0, \sqrt{t+\varepsilon_0}) \) in \eqref{growth-grap}, which holds for \( t \geq r_0^2 - \varepsilon_0 \), we deduce that
 \begin{align*}
        \begin{aligned}
            [S_G(t)u_0](x) &=  \sum_{y \in V} \mu(y) p_G(t,x,y) u_0(y)\\
            & = \lambda \sum_{y \in V} \mu(y) p_G(t,x,y) p_G(\varepsilon_0, y, x_0) \\
            & = \lambda \, p_G(t+\varepsilon_0, x, x_0)\\
            & \leq \lambda \,  \dfrac{C_1 }{V_\mu(x_0, \sqrt{t+\varepsilon_0})} \\
            & \leq \lambda \, C \, (t+\varepsilon_0)^{-d/2} [\ln (1+ \sqrt{t+\varepsilon_0})]^{\xi},
        \end{aligned}
    \end{align*}
for any $t > \max\{ t_0, r_0 ^2 - \varepsilon_0\} \equiv m$ and $x\in V$. So 
$$\|  S_G(t)u_{0} \|_{L^{\infty}(V)} \leq \lambda \, C \, (t+\varepsilon_0)^{-d/2} [\ln (1+ \sqrt{t+\varepsilon_0} )]^{\xi},$$
for $t>m$. Then
\begin{align*}
& \sum_{i=1}^{k}    \int_{0}^{\infty}  \ell_i(\sigma)  F_i(\|S_{\mathbb{G}}(\sigma)w_{0}  \|_{\infty})  d\sigma \\
   & = \displaystyle\sum_{i=1}^k \int_{0}^{\infty}  \sigma^{\alpha_i} \|  S_G(\sigma)u_{0}  \|_{L^{\infty}(V)}^{q_i-1}  d\sigma  \\
   \noalign{\medskip}
   &= \displaystyle\sum_{i=1}^k \int_0^{m}     \sigma^{\alpha_i} \|  S_G(\sigma)u_{0}  \|_{L^{\infty}(V) }^{q_i-1}  d\sigma\\
    & + \displaystyle\sum_{i=1}^k \int_{m}^{\infty}     \sigma^{\alpha_i} \|  S_G(\sigma)u_{0}  \|_{L^{\infty}(V) }^{q_i-1}  d\sigma \\
    \noalign{\medskip}
    & \leq \left(\frac{1}{n}\right)^{q_i-1} \displaystyle\sum_{i=1}^k \int_0^{m}     \sigma^{\alpha_i}   d\sigma \\
    & + \lambda^{q_i-1} \displaystyle\sum_{i=1}^k  \int_{m}^{\infty} {\sigma}^{ \frac{-d(q_i-1)}{2} + \alpha_i} [\ln (1+ \sqrt{\sigma+\varepsilon_0})]^{\xi (q_i-1)}  d\sigma.     
\end{align*}
The last $k$ integrals above converge since $ -{d(q_i-1)}/{2} + \alpha_i <-1$ for all $i=1, \cdots,k$, i.e, $q_i< 1 + 2( \alpha_i+1)/d$ for all $i=1,\cdots,k.$ Thus, choosing $n$  sufficiently large (therefore, $\lambda$  sufficiently small) we obtain:
$$ \displaystyle\sum_{i=1}^k \int_{0}^{\infty}  \sigma^{\alpha_i} \|  S_G(\sigma)u_{0}  \|_{L^{\infty}(V)}^{q_i-1}  d\sigma < \frac{\beta}{\beta +1}, $$
for some $\beta>0$. Therefore, the condition in item (i) of Theorem 1 is satisfied, and the result follows.
\end{proof}

\begin{remark} By applying Corollary \ref{co-graf1} (i) with $i=1$, we obtain the Fujita critical exponent $q^* = 1 + {2(\alpha+1)}/{d}$. This exponent is also identified in the works \cite[Theorem 3.3]{Lin1} (where $\eta=0$) and \cite[Theorem 2.10]{LIU}.
\end{remark}

 \begin{remark} Similar arguments to those used in the proof of Corollaries \ref{coEXTRA} (i) and \ref{coEXP} can be used to treat the case $\ell_i(t) = e^{\alpha_i t}$, $\alpha_i>0$ (instead of $\ell_i(t) = t^{\alpha_i}$, $\alpha_i>-1 $). In this case, there are no global solutions for problem \eqref{eq2.1}. On the other hand, if \(\ell_i(t) = t^{\alpha_i}\), \(\alpha_i > -1\) (instead of $\ell_i(t) = e^{\alpha_i t}$, $\alpha_i > 0$), the proof of Corollary \ref{co2} (ii) shows that problem \eqref{eq2.1} has a global solution for any $q_i > 0$. In particular, the case  $i = 1$ recovers \cite[Theorem 1.1 - (1)]{BANDLE}.
\end{remark}

\begin{remark}[\textbf{Combined nonlinearities}] In view of the Corollaries \ref{coEXTRA}, \ref{co2} and \ref{co-graf1} (i), we can conclude that the result remains unchanged if there is a combination of polynomial and logarithmic nonlinearities $f_i(u) = u^{q_i}$, $q_i>1$ for $i=1, \cdots, k-l$ and $f_i(u) = (1+u)[\ln(1+u)]^{q_i}$, $q_i>1$ for $i =k-l+1, \cdots, k$. However, if there is a combination with an exponential nonlinearity $f_i (u) = e^{ \beta_i u} -1$, $\beta_i>0$ for some $i \in \{1, \cdots, k\} $, the result of the Corollaries \ref{coEXP}, \ref{coEXP-hiperbolico} and \ref{co-graf1}  (ii) prevails, and there are no global solutions.
\end{remark}

\backmatter
\vspace{1em}
\bmhead{Conflict of interest}
The authors declare that they have no competing interests.

\backmatter
\vspace{1em}
\bmhead{Data availability}
Data sharing do not apply to this article as no datasets were generated or analyzed during the current study.


\backmatter

\vspace{1em}



\bmhead{Acknowledgments}
R. Castillo is supported by ANID-FONDECYT project No. 11220152. R. Freire is supported by CNPq (Conselho Nacional de Desenvolvimento Científico e Tecnológico) Process 140600/2020-5. M. Loayza is partially supported by MATH-AMSUD project 21-MATH-03 (CTMicrAAPDEs), CAPES-PRINT 88887.311962/2018-00 (Brazil), CNPq - 313382/2023-9.

\vspace{1em}





\noindent

\end{document}